\documentclass{article}
\usepackage{amsthm}
\usepackage{amsmath}
\usepackage{amssymb}
\usepackage{epstopdf}
\usepackage{graphicx}
\newtheorem{theorem}{Theorem}

\begin{document}
\title
{Towards a resolution of the Riemann hypothesis}
\author{R.C. McPhedran,\\
School of Physics,\\
University of Sydney}
\maketitle
\begin{abstract}
This article contains work associated with  a resolution of the Riemann hypothesis, following work by Taylor \cite{prt}, Lagarias and Suzuki \cite{lagandsuz} and Ki \cite{ki}, as well as Pustyl'nikov \cite{pust, pust2} and Keiper \cite{keiper}.
Functions $\xi_+(s)$ and $\xi_-(s)$  are considered, for which it is known that all zeros lie on the critical line. The Riemann hypothesis itself pertains to the question of the location of the zeros of the sum and difference of $\xi_+(s)$ and $\xi_-(s)$, and this is investigated.
An argument is developed which prima facie establishes the validity of the Riemann hypothesis. It adds to a necessary condition of Pustyl'nikov a sufficient condition. A second argument is discussed, which could have been accessible to Riemann.
\end{abstract}

\section{Introduction}
The Riemann hypothesis is widely recognised as one of the most important and difficult unsolved problems in mathematics \cite{edw, titheath}. It asserts that all non-trivial zeros of the function $\zeta (s)=\zeta (\sigma+ i t)$ lie on the critical line $\sigma=1/2$.  Since its formulation in 1859 it has inspired the creation of a vast and deep literature, both analytic and numeric.

It is not a purpose here to review this panoply of scholarship, but rather to follow elements of three particular streams concerning functions related to $\zeta (s)$ for which is has been proved that all zeros in fact lie on the critical line. The first stream commenced with a posthumous paper by Taylor \cite{prt}, which proved that the difference of two $\xi$ functions (symmetrised zeta functions) whose argument differed by unity had all its zeros on the critical line. It was continued around sixty years later by papers of Lagarias and Suzuki \cite{lagandsuz} and Ki \cite{ki}, which showed that the sum of the same functions had all its zeros on the critical line.

The second stream derives from the work of Pustyl'nikov \cite{pust, pust2}. He derived an integral form for the derivatives of one of the two $\xi$ functions studied by the previous group of authors. He established that all these derivatives were positive reals, and that this was a necessary condition for the Riemann hypothesis to hold. Here we establish a convenient power series for the second $\xi$ function, and series for the symmetric and antisymmetric combinations of the $\xi$ series, as well as bounding relations among these functions along the real axis of $s$.

The third stream is concerned with relations between the coefficients of power series of $\xi$ and its logarithm and sums over inverse powers of the zeros of $\xi$ \cite{lehmer, keiper, li,bomblag}. These relations enable necessary conditions to be established for the 
Riemann hypothesis to hold, and to do this use a mapping from the critical line to the unit circle.

The outline of the rest of this article is as follows. In Section 2, we give definitions and basic properties of the functions we deal with.
We also consider mappings for zeros of functions onto straight lines and circles illustrated in Fig. \ref{fig-regions}, with the symmetric and antisymmetric combinations of $\xi$ functions having all their zeros on the unit circle. The Riemann hypothesis requires then the zeros of the two $\xi$ functions to lie on a circle of radius one half and a straight line: these touch at the rightmost point of the unit circle in the transformed variable $w$. Section 3 discusses power series for the various functions $\xi$ and their combinations, both in $s$ and the transformed variable $w$. The sums over inverse powers of zeros are connected with these series, and bounding relations are established for real $w$. Section 4 gives the main result of this paper, which is that Pustyl'nikov's necessary condition is also sufficient- a prima facie proof of the Riemann hypothesis.  The argument is based on establishing the radius of convergence of a series for $\log \xi$, given the fact that the symmetric and antisymmetric functions obey the Riemann hypothesis, and the bounds on the real $w$ axis. Section 5 contains a parallel argument, relying on elements probably known to Riemann. Section 6 exhibits power series of various functions discussed in this paper, up to order 12 and with coefficients to  24 decimals.

Note that all calculations for this paper have been performed on a laptop (MacBook Pro, late 2013, 16 GB) in Mathematica, and thus are readily checked by those having computers of moderate power and adequate software.

\section{Definitions}
The function $\xi(s)$ is even under $s\rightarrow 1-s$ and is defined as
\begin{equation}
\xi(s)=\frac{1}{2}  s(s-1) \frac{\Gamma (s/2)\zeta (s)}{\pi^{s/2}}=\frac{1}{2}  s(s-1) \xi_1(s).
\label{xi1}
\end{equation}
We will also be interested in the following two combinations of $\xi$:
\begin{equation}
\xi_+(s)=\frac{1}{2}\left[\xi\left( s+\frac{1}{2}\right)+\xi\left( s-\frac{1}{2}\right)\right],
\label{bpust5}
\end{equation}
and 
\begin{equation}
\xi_-(s)=\frac{1}{2}\left[\xi\left( s+\frac{1}{2}\right)-\xi\left( s-\frac{1}{2}\right)\right].
\label{bpust6}
\end{equation}
These are respectively even and odd under the substitution $s\rightarrow 1-s$.
The fact that all the non-trivial zeros of the function $\xi_-(s)$ lie on the critical line was first established by P.R. Taylor \cite{prt}. The location of 
all the non-trivial zeros of the function $\xi_+(s)$ on the critical line was established by Lagarias and Suzuki \cite{lagandsuz}, Ki \cite{ki}, and McPhedran and Poulton \cite{mcp13}.

We will be interested in mappings which move the location of lines along which zeros are located onto circles in the complex plane.
For the case of  $\xi_-(s)$ and  $\xi_+(s)$ the mapping from the  critical line $\Re (s)=\sigma=1/2$ onto the unit circle is
\begin{equation}
w=u+i v=1-\frac{1}{s}=\frac{s-1}{s}.
\label{tr1}
\end{equation}
The inverse transformation is
\begin{equation}
s=\frac{1}{1-w}.
\label{tr2}
\end{equation}

For the function $\xi\left( s+1/2\right)$ the corresponding forward transformation is
\begin{equation}
w_h=1-\frac{1}{s+1/2}=\frac{s-1/2}{s+1/2}.
\label{tr3}
\end{equation}
Its  inverse transformation is
\begin{equation}
s=-\frac{1}{2}+\frac{1}{1-w_h}.
\label{tr4}
\end{equation}
The zeros of $\xi\left( s+1/2\right)$ lie on $\sigma=0$ and are mapped onto the unit circle in the plane of complex $w_h$.

For the function $\xi\left( s-1/2\right)$ the corresponding forward transformation is
\begin{equation}
w_m=1-\frac{1}{s-1/2}=\frac{s-3/2}{s-1/2}.
\label{tr5}
\end{equation}
Its  inverse transformation is
\begin{equation}
s=\frac{1}{2}+\frac{1}{1-w_m}.
\label{tr6}
\end{equation}
The zeros of $\xi\left( s-1/2\right)$ lie on $\sigma=1$ and are mapped onto the unit circle in the plane of complex $w_m$.

We next connect $w_h$ and $w_m$  to the complex variable $w$. Eliminating $s$ between (\ref{tr1}) and (\ref{tr4}) we find
\begin{equation}
w_h=\frac{1+w}{3-w}, ~~w=\frac{3 w_h-1}{w_h+1}.
\label{tr7}
\end{equation}
The equation for the fixed point of this transformation is
\begin{equation}
w (3-w)=1+w, ~{\rm or}~ (w-1)^2=0.
\label{tr8}
\end{equation}
The fixed point is then of second order at $w=w_h=1$. The corresponding equations relating to $w_m$ are
\begin{equation}
w_m=\frac{3 w-1}{w+1}, ~~w=\frac{1+w_m}{3-w_m},
\label{tr9}
\end{equation}
while again the fixed point for the transformation yields $w=1=w_m$ being of second order.

Figure \ref{fig-regions} illustrates curves of constant modulus pertaining to these discussions. The black unit circle corresponds to $|w|=1$, and thus to the mapping of the critical line through equation (\ref{tr1}). The blue circle centred on $w=1/2$ of radius $1/2$ corresponds to the constraint $|w_m|=1$, and the mapping of the line $\sigma=1$, with the blue vertical line being $|w_m|=3$, $u=-1/3$.  The red lines are for $|w_h|=1$ and $|w_h|=2$. The two fixed points occur where the two circles touch, with the line $|w_h|=1$ being tangent to both.

\begin{figure}[tbh]
\includegraphics[width=7.5 cm]{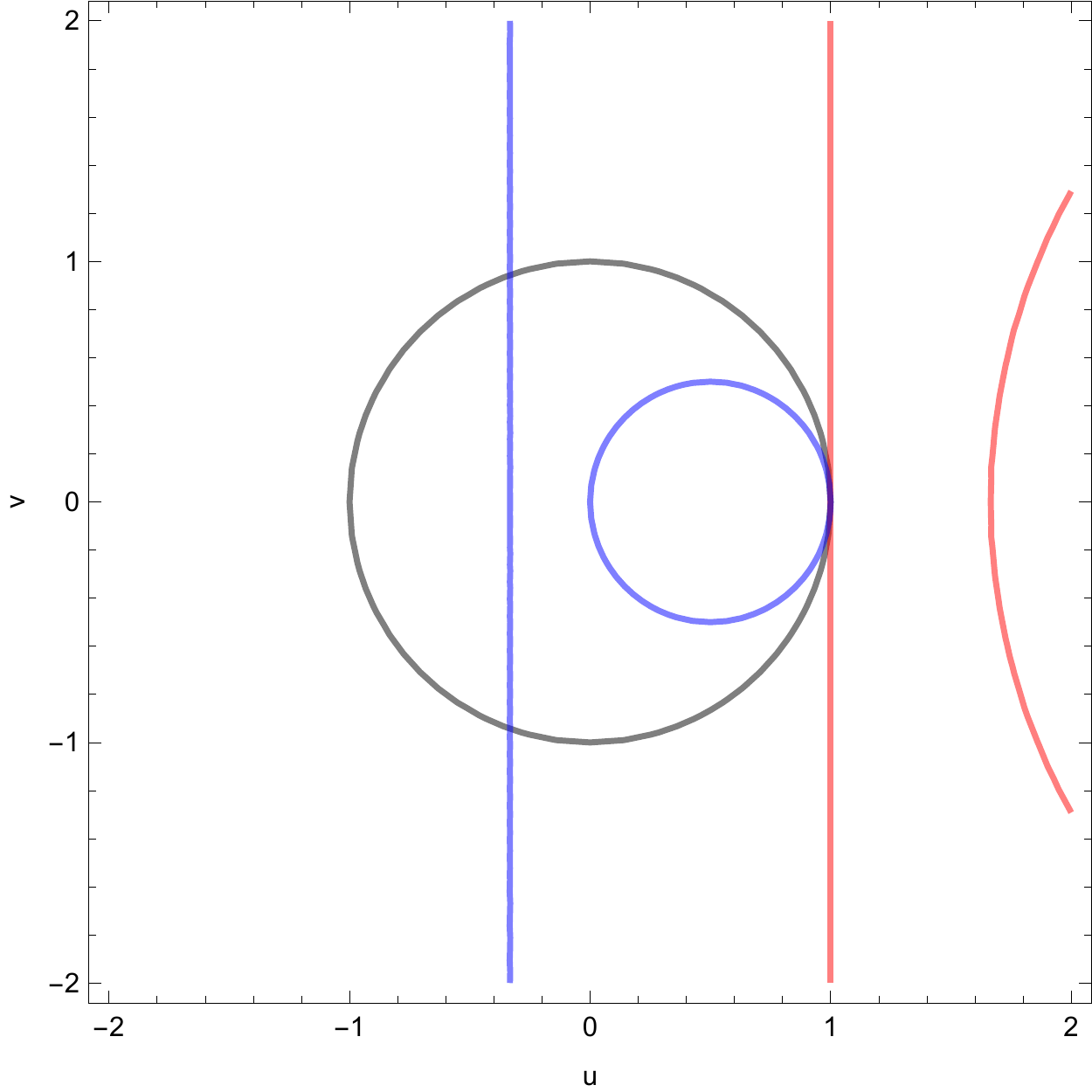}
\caption{ The transformations pertaining to $w$ (black circle), $w_m$ (blue circle and blue line) and $w_h$ (red lines).}
\label{fig-regions}
\end{figure}

We give in Table 1 values for $w$, $s$, $\xi(s+1/2$ and $\xi(s-1/2)$ at key points in Fig. \ref{fig-regions}.
\begin{table}
\label{tab1}
\begin{center}
\begin{tabular}{|c|c|c|c|}\hline
$w$ & $s$ & $\xi(s+1/2)$ &$\xi(s-1/2)$ \\ \hline 
 -1 & $1/2$ & $1/2$& $1/2$\\
 $-1/3$ & $3/4$ & $0.503621$ & $0.497839$ \\
 $0$ & $1$ & $0.508371$ & $0.497121$\\
 $1/2$ & $2$ & $0.545094$ & $0.508731$ \\
 $0.9$ & $10$ & $4.31356$ & $2.9175$ \\
 $0.95$ & $20$ & $1024.78$ & $531.726$ \\  \hline
\end{tabular}
\caption{ Data for key points in Fig. \ref{fig-regions}.}
\end{center}
\end{table}
Note the equality of $\xi(s+1/2)$  and $\xi(s-1/2)$ for $s=1/2$ means that $\xi_{-}(1/2)$ is zero. Also, the last two rows indicate that 
  $\xi(s+1/2)$ and $\xi(s-1/2)$ diverge faster than at a simple pole as $w\rightarrow 1$.         
\section{Expansions of $\xi$}
We commence with the expansion of $\xi (s+1/2)$, based on the work of Pustyl'nikov \cite{pust, pust2}:
  \begin{equation}
 \xi \left( s+\frac{1}{2}\right)=\sum_{r=0}^\infty \xi_r s^{2 r},
 \label{bpust1}
 \end{equation}
 where the coefficients $\xi_r$ can be obtained in integral form from
  \begin{equation}
 \xi_{r}=\frac{2^{-(2 r+2)}}{(2 r)!} {\cal I}_{2 r},
 \label{pust4}
 \end{equation}
 with
 \begin{equation}
{\cal I}_r=  \int_1^\infty [\log (x)^{r-2} ][16 r (r-1)-\log(x)^2] x^{-3/4} \omega(x) d x
 \label{pust4a}
 \end{equation}
and $\omega (x)=\sum_{n=1}^\infty \exp (-\pi n^2 x)$ being related to the  elliptic  theta function $\vartheta_3$.  Pustyl'nikov \cite{pust2}
establishes in his Theorem 1 that all even order derivatives of $\xi (s)$ at $s=1/2$ are strictly positive. His Theorem 2 is that Theorem 1 provides a necessary condition for the Riemann hypothesis to hold. Asymptotic analysis of the integral (\ref{bpust1}) may be found in
Pustyl'nikov \cite{pust}, and supplementary comments  with numeric results are available in \cite{rmcp19}.

We replace $s+1/2$ by $s$ on the left-hand side of (\ref{bpust1})  and expand the consequent right-hand side terms in $(s-1/2)$ using the Binomial Theorem. After collecting  powers of $s$ we arrive at:
  \begin{equation}
 \xi (s)=\sum_{n=0}^\infty  \left[ \sum_{r\ge n/2}^\infty  \binom{2 r}{n} \xi_r \left(\frac{-1}{2}\right)^{2 r-n} \right]  s^{n}.
  \label{bpust2}
 \end{equation}
 Replacing $s$ by $-s$ in (\ref{bpust2}) and using the even symmetry of $\xi(s)$ under $s\rightarrow 1-s$ we find:
  \begin{equation}
 \xi (1+s)=\sum_{n=0}^\infty  \left[ \sum_{r\ge n/2}^\infty  \binom{2 r}{n} \xi_r \left(\frac{1}{2}\right)^{2 r-n} \right]  s^{n}.
  \label{bpust3}
 \end{equation}
 We can also replace $s$ by $s-1$ in (\ref{bpust1}) to give:
  \begin{equation}
 \xi \left( s-\frac{1}{2}\right)=\sum_{n=0}^\infty \left[\sum_{2 r\ge n}^\infty \binom{2 r}{n} \xi_r\right]  (-1)^n s^{n}.
 \label{bpust4}
 \end{equation}
 Separating the expression in (\ref{bpust4}) into the contributions from odd and even $n$:
  \begin{equation}
 \xi \left( s-\frac{1}{2}\right)=\sum_{n=0}^\infty \left[\sum_{ r= n}^\infty \binom{2 r}{2 n} \xi_r\right]   s^{2 n}
 -\sum_{n=0}^\infty \left[\sum_{ r= n+1}^\infty \binom{2 r}{2 n+1} \xi_r\right]   s^{2 n+1}.
 \label{bpust4a}
 \end{equation}
 
The expansions of  the functions $\xi_+(s)$ and $\xi_-(s)$ are then:
\begin{equation}
\xi_+(s)=\frac{1}{2}\left\{2 \sum_{n=0}^\infty \xi_n s^{2 n}-\sum_{n=0}^\infty \sum_{r=n+1}^\infty \left[s\binom{2 r}{2 n+1}-\binom{2 r}{2 n}\right]\xi_r s^{2 n}\right\} ,
\label{bpust7}
\end{equation}
and
\begin{equation}
\xi_-(s)=\frac{1}{2}\left\{\sum_{n=0}^\infty \sum_{r=n+1}^\infty \left[s\binom{2 r}{2 n+1}-\binom{2 r}{2 n}\right]\xi_r s^{2 n}\right\} .
\label{bpust8}
\end{equation}

For use in the various expansions of this section, the first 5000 constants $\xi_r$ have been evaluated with an accuracy goal of 24 decimals in Mathematica \cite{rmcp19}.

\subsection{Connection of power series with sums over powers of zeros} 

If $f(z)$ is an analytic function having an infinite set of  simple zeros at the points $a_0$, $a_1$, $a_3$, $\ldots$ which tend to infinity as $n\rightarrow \infty$, with $|a_n|\neq 0$ for all $n$, then we have the sum rules for positive integers $m$:
\begin{equation}
\left. \frac{1}{m!}\frac{d^m}{d z^m} \log f(z)\right|_{z=0} =-\frac{1}{m}\sum_n\frac{1}{a_n^m}.
\label{prod6}
\end{equation}
We can find the Taylor series coefficients of $\log [\xi(s)/\xi(0)]$ by employing logarithmic polynomials \cite{qietal}. Denoting the zeros of
$\xi(s)$ by $\rho$, we obtain values for the sums over inverse powers of zeros:
\begin{equation}
\sigma_k=\sum \frac{1}{\rho^k},
\label{prodk1}
\end{equation}
which are always real since every zero $\rho$ can be paired with its conjugate, also a zero. 
Keiper also consider the relationship between the coefficients $\sigma_k$ and those occurring in two further expansions:
\begin{equation}
\frac{\xi'(1/s)}{\xi(1/s)}=\sum_{k=0}^\infty \tau_{k} (1-s)^k ,
\label{ke6}
\end{equation}
and
\begin{equation}
\log (2\xi(1/s)) =\sum_{k=0}^\infty \lambda _{k} (1-s)^k.
\label{ke7}
\end{equation}
He shows that:
\begin{equation}
\tau_0=\sigma_1,
\label{ke8}
\end{equation}
and
\begin{equation}
\tau_k=\sum_{j=1}^k \left(\begin{tabular}{c}
$k-1$\\
$j-1$
\end{tabular}\right) (-1)^j \sigma_{j+1} ~{\rm for} ~ k\ge 1.
\label{ke9}
\end{equation}
Also,
\begin{equation}
\lambda_0= 0,
\label{ke10}
\end{equation}
and
\begin{equation}
\lambda_k=\sum_{j=1}^k \frac{(-1)^{j-1}}{j} \left(\begin{tabular}{c}
$k-1$\\
$j-1$
\end{tabular}\right) \sigma_{j} ~{\rm for} ~ k\ge 1.
\label{ke11}
\end{equation}

The coefficients $\tau_k$ hold a particular interest in relation to the Riemann hypothesis. Indeed, as Keiper shows,
\begin{equation}
\tau_{m-1}=-\sum_\rho  \left(\frac{\rho}{\rho-1}\right)^m \rho^{-2}.
\label{ke12}
\end{equation}
From (\ref{ke12}), if the Riemann hypothesis holds, the $|\tau_k|$ must be bounded by 
\begin{equation}
\sum_\rho |\rho|^{-2}=0.046191479322\ldots .
\label{ke13}
\end{equation}
On the other hand, if the $|\tau_k|$ are bounded, then for no $\rho$, $|\rho|>|1-\rho|$, so the Riemann hypothesis holds.

In terms of the behaviour of the $\lambda_k$, Li's criterion \cite{li, bomblag} states that the Riemann hypothesis is equivalent to 
$\lambda_k\ge 0$ for every positive integer $k$. These two criteria are illustrated  in Fig. \ref{fig-xi1}. The quantities $\tau_k$ decrease as $k$ increases in the range shown, moving further below the limit $0.046191479322$. The quantities $\lambda_k$ increase roughly linearly with $k$ (the slope being about $0.023$), again moving away from the limit of zero. Keiper comments on the difficulty of finding numerically exceptions to the Riemann hypothesis using this sort of behaviour.
 Indeed, if we consider the equation
 \begin{equation}
 \lambda_m=\frac{1}{m}\sum_\rho \left[ 1- \left(\frac{\rho}{\rho-1}\right)^m\right],
 \label{ke14}
 \end{equation}
 then for the quantity in square brackets to become negative for an exception to the Riemann hypothesis with $t>T$,
 we require (roughly) $m>2 T^2$. Currently, $T=O(10^9)$, so $m>O(10^{18})$.
\begin{figure}[tbh]
\includegraphics[width=6 cm]{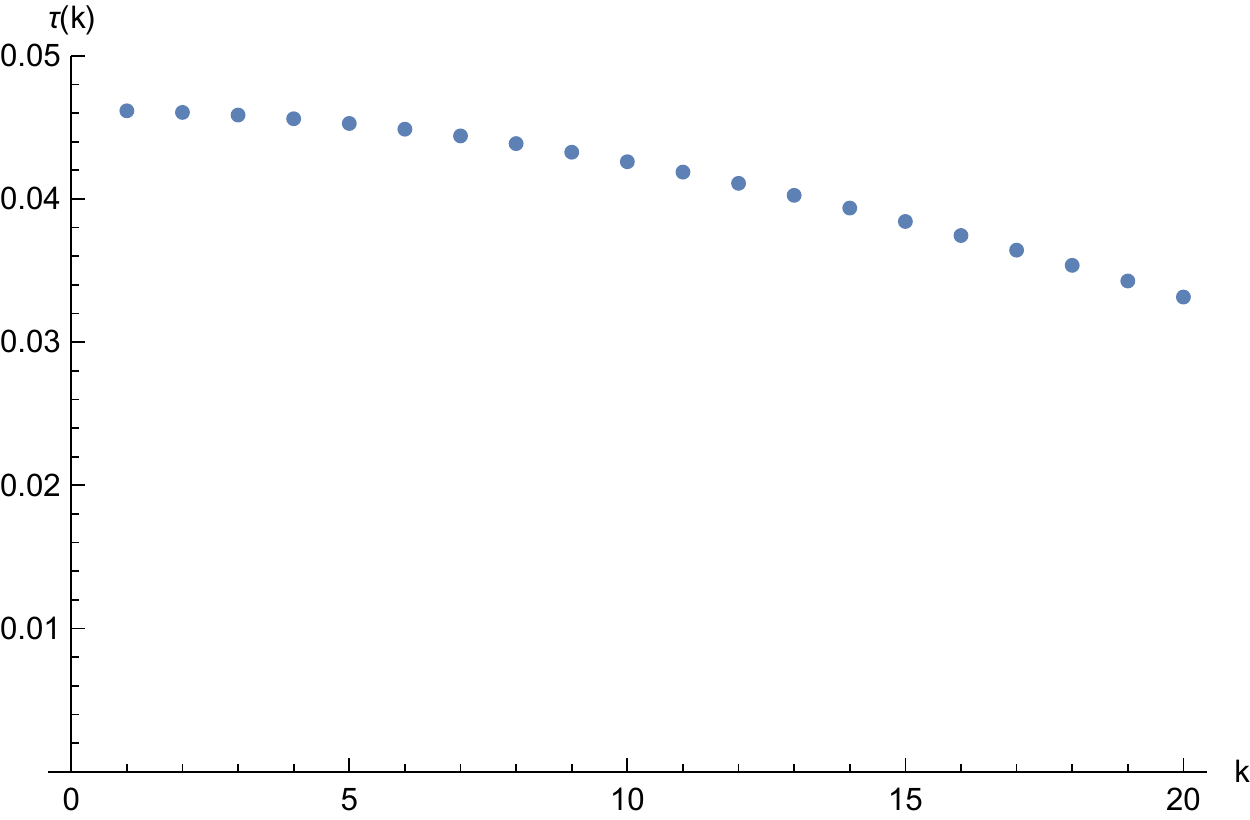}~~\includegraphics[width=6 cm]{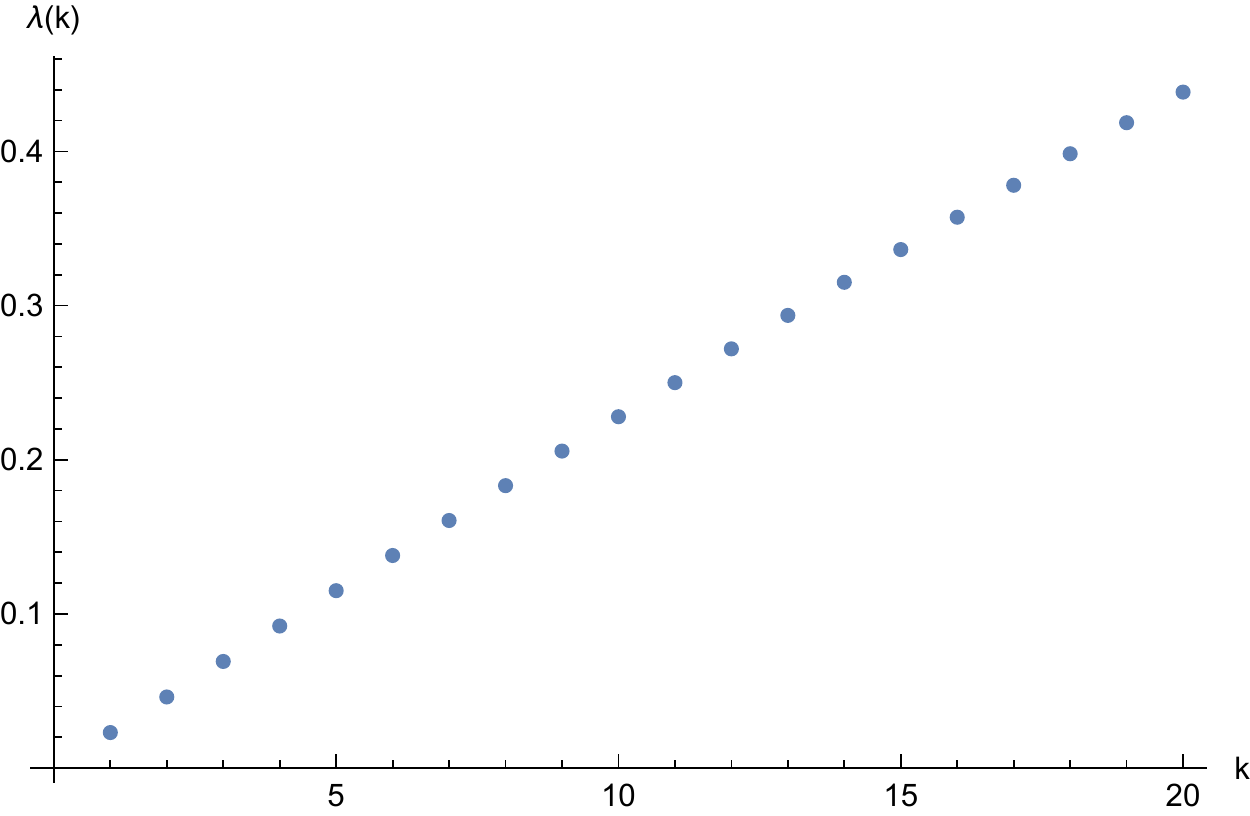}
\caption{(Left)} The coefficients $\tau (k)$ of $\xi (s)$ as a function of their order $k$. (Right) The coefficients $\lambda (k)$ as a function of $k$.
\label{fig-xi1}
\end{figure}
\subsection{Power series in the variable $w$}

The expansions we will now exhibit are those relevant to  Fig. \ref{fig-regions}, namely: for $\xi (s+1/2)$, $\xi (s-1/2)$,
$\xi_+(s)$ and $\xi_-(s)$ with $s$ replaced by $1/(1-w)$ about $w=0$ pertaining to the unit circle in the $w$ plane, and for
$\xi(s-1/2)$ with $s$ replaced by $1/(1-\hat{w})$ pertaining to the circle of radius $1/2$ centred on $\hat{w}=w-1/2=0$. 

From (\ref{bpust1}),
  \begin{equation}
 \xi \left( \frac{1}{1-w}+\frac{1}{2}\right)=\sum_{r=0}^\infty \xi_r \left( \frac{1}{1-w}\right)^{2 r}.
 \label{bpust1w}
 \end{equation}
 Here from Theorem 1 of Pustyl'nikov  \cite{pust2} all the $\xi_r$ are positive, and $\xi_0 \approx 0.497121$. This function is then monotonic increasing as $w$ increases from negative infinity up to $w=1$, where it encounters an essential singularity, with all its derivatives tending to infinity.

 We can also find the expansion for $\xi (s-1/2)$ using (\ref{bpust1}):
 \begin{equation}
  \xi \left( \frac{1}{1-w}-\frac{1}{2}\right)= \xi \left( \frac{w}{1-w}+\frac{1}{2}\right)=\sum_{r=0}^\infty \xi_r \left( \frac{w}{1-w}\right)^{2 r}.
  \label{bpust1wa}
 \end{equation}
 Once again, the form (\ref{bpust1wa}) is monotonic increasing as $w$ increases from negative infinity up to $w=1$, where it encounters an essential singularity, with all its derivatives tending to infinity.
 
 Combining (\ref{bpust1w}) and (\ref{bpust1wa}), we find that
 \begin{equation}
 \xi_+ \left( \frac{1}{1-w}\right)=\frac{1}{2}\sum _{r=0}^\infty \xi_r  \left[\frac{(1+ w^{2 r})}{(1-w)^{2 r}}\right],
 \label{bpust2w}
 \end{equation}
 and
  \begin{equation}
 \xi_- \left( \frac{1}{1-w}\right)=\frac{1}{2}\sum _{r=0}^\infty \xi_r  \left[\frac{(1- w^{2 r})}{(1-w)^{2 r}}\right].
 \label{bpust2wa}
 \end{equation}
 The leading terms as $w\rightarrow 1$ in these two expansions are respectively $\xi_r/(1-w)^{2 r}$ and $r \xi_r/(1-w)^{2 r-1}$, so the former will be much larger than the latter- see Fig \ref{figcfxi}. Both will increase monotonically for increasing real $w$.
 
 From (\ref{bpust1}-\ref{bpust2wa}), for real $w<1$ we have the inequalities:
 \begin{equation}
 \xi\left( \frac{1}{1-w}+\frac{1}{2}\right)> \xi_+ \left( \frac{1}{1-w}\right)>\xi_- \left( \frac{1}{1-w}\right)>0 ,
 \label{ineq1}
 \end{equation}
 and
 \begin{equation}
  \xi_+ \left( \frac{1}{1-w}\right)>\xi \left( \frac{1}{1-w} -\frac{1}{2}\right)>0.
 \label{ineq2}
 \end{equation}
  
\begin{figure}[tbh]
\includegraphics[width=7.5 cm]{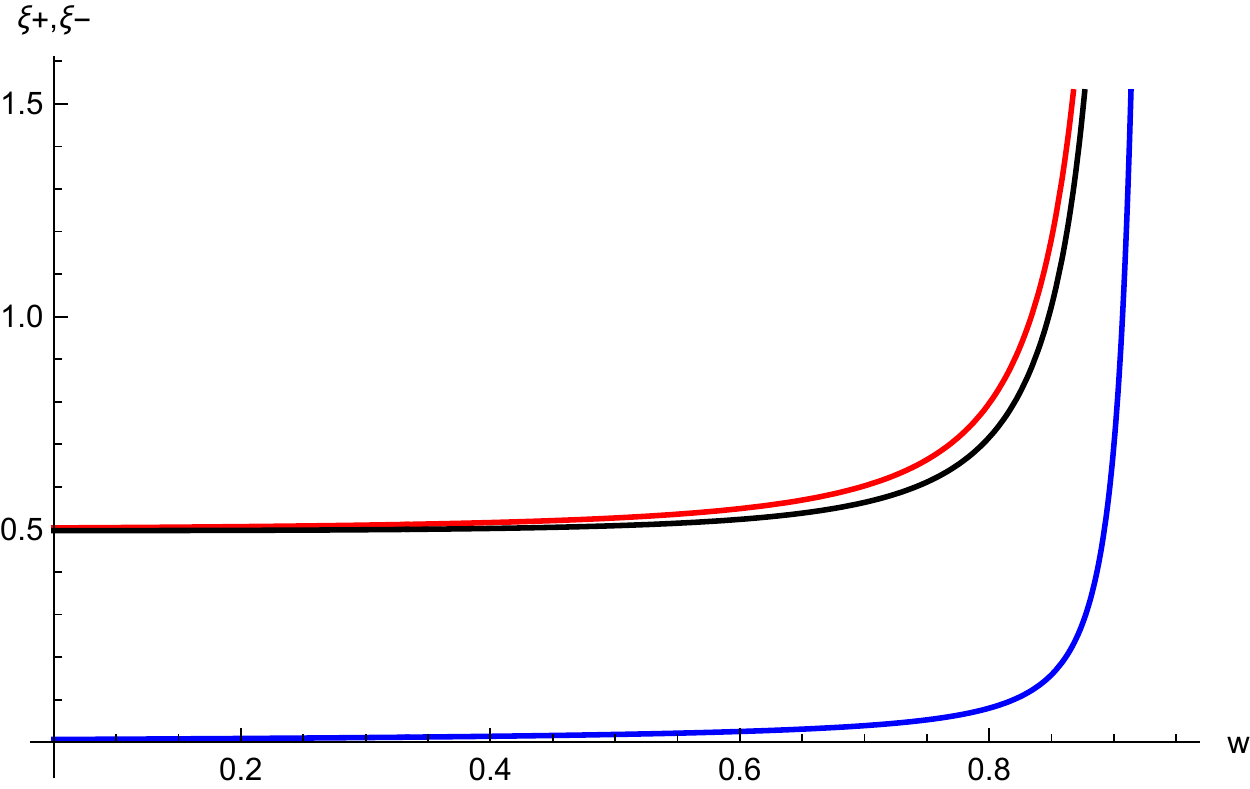}~\includegraphics[width=7.5 cm]{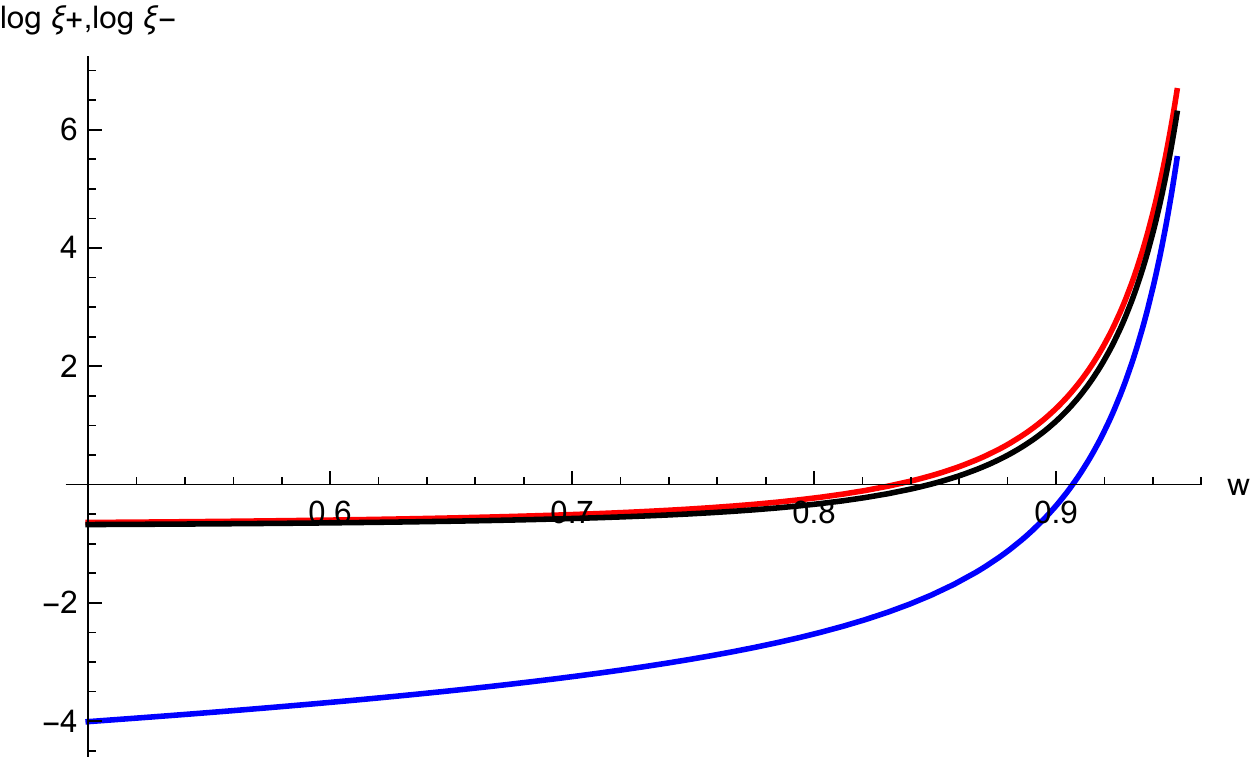}
\caption{(Left) The functions $\xi_+$ (red) and $\xi_-$ (blue) as a function of real $w$.The black curve corresponds to $\xi(s-1/2)$, and to the difference of the red and blue curves.  (Right) The corresponding plot of logarithms of the functions.}
\label{figcfxi}
\end{figure} 
\section{Extensions of Pustyl'nikov's results}
We have already mentioned theorems due to Pustyl'nikov: all the coefficients $\xi_r$ are positive, and that this is a necessary condition for the Riemann hypothesis to hold. The aim of this section is to provide a proof that the condition is also sufficient.
\begin{theorem}
If all the coefficients $\xi_r$ are positive, then all zeros of $\xi (s-1/2)$ lie on the line $\sigma=1$.
\label{suffcond}
\end{theorem}
\begin{proof}
We know from the results of Taylor \cite{prt}, Lagarias and Suzuki  \cite{lagandsuz} and Ki \cite{ki} that all zeros of $\xi_+(s)$ and $\xi_-(s)$ lie on $\sigma =1/2$, which corresponds to the unit circle in the $w$ plane. The radii of convergence of the power series of  $\log \xi_+(1/(w-1))$ and $\log\xi_-(1/(w-1))$ about $w=0$ are then unity. The singularities of these two logarithmic functions lie at the zeros of $\xi_+(s)$ and $\xi_-(s)$ and at the essential singularity of each, which is at $w=1$. They can be re-expanded about $w=1/2$, to give power series which have a radius of convergence of $1/2$, diverging at the essential singularity.

Our objective is to prove that the radius of convergence of the power series of $\log \xi(1/(w-1)-1/2$ about $w=1/2$ is 1/2: this corresponds to the inner circle of Fig. \ref{fig-regions}. The essential singularity of this function lies at $w=1$, and from the fact that an infinite number of zeros of $\zeta (s)$ lies on $\sigma =1/2$, we know that these lie on the inner circle. The radius of convergence can thus not exceed unity. Now, by the theory of power series \cite{churchill}, they converge absolutely everywhere inside the circle of convergence, and diverge everywhere outside it. Thus, all we need to prove is that the power series for $\log \xi(1/(w-1)-1/2)$ about the centre $w=1/2$ converges everywhere along the ray from $w=1/2$ with $w$ real, for $w<1$. However, this is trivially true, since we have shown in the inequality (\ref{ineq2}) that the power series for  $\xi(1/(w-1)-1/2)$ lies between that  for $\xi_+(1/(w-1))$and zero, while the series for $\xi_+(1/(w-1))$   converges up to $w=1$. Consequently, the power series for   $\log \xi(1/(w-1)-1/2)$ lies below that for  $\log \xi_+(1/(w-1))$  up to $w=1$ (by the monotonicity of the logarithmic function). Both logarithmic power series are monotonic increasing, and so the lower one cannot diverge before the upper one, while both diverge at $w=1$.  The radius of convergence of the series for $\log \xi(1/(w-1)-1/2)$ must then be precisely one half.
\end{proof}
\section{Riemann's reasoning?}
In this section, we present an alternative argument  to the Theorem of the previous section, which does not rely on the results of Pustyl'nikov, and
is built from ideas probably accessible to Riemann. We consider the expansion of $\log \xi(s-1/2)$:
\begin{equation}
\log \xi \left(s-\frac{1}{2}\right)=\log\left[ \frac{1}{2} \left( s-\frac{1}{2}\right)  (s-1) \frac{\Gamma \left( \frac{1}{2} \left( s-\frac{1}{2}\right)\right)}{\pi^{\left( \frac{1}{2} \left(s-\frac{1}{2}\right)  \right)}}\right]+\log\zeta \left(s-\frac{1}{2}\right) .
\label{riereas1}
\end{equation}
We treat the two parts on the right separately. We replace in each $s-1/2$ by $1/(1-w)-1/2$ and expand for $w$ approach unity from below. The series for the prefactor is
\begin{equation}
\frac{\log [2 e \pi (1-w)]}{2 (w-1)}+\frac{1}{4} \left( \log \left(2\pi^3 \right)-5\log (1-w)\right)+\frac{73}{48} (w-1)+O(w-1)^2.
\label{riereas2}
\end{equation}
The series for $\zeta$ is absolutely convergent for $w$ real, and is in truncated form
\begin{equation}
\log \zeta \left( \frac{1+w}{2(1-w)}\right)\approx\log \left[1+2^{-\frac{1+w}{2(1-w)}}+3^{-\frac{1+w}{2(1-w)}}+4^{-\frac{1+w}{2(1-w)}}\right].
\label{riereas3}
\end{equation}
The approximation (\ref{riereas3}) has a relative error of 2.5\% at $w=0.8$, 0.5\% at $w=0.85$ and 0.02\% at $w=0.9$.

\begin{figure}[tbh]
\includegraphics[width=7.5 cm]{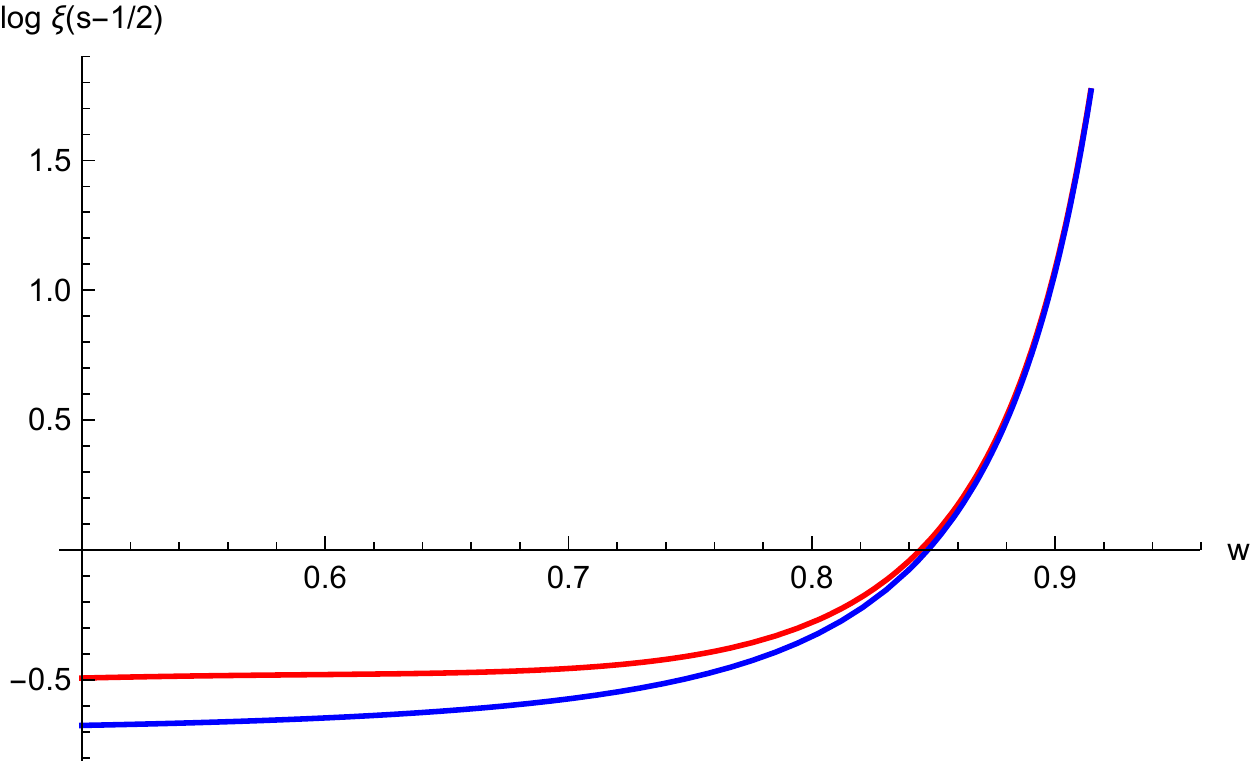}
\caption{ The blue curve gives the value of $\log \xi(1/(1-w)-1/2)$ and the red curve shows the result of combining (\ref{riereas2}) and  (\ref{riereas3}).}
\label{fig-xiapprox}
\end{figure}

Fig. \ref{fig-xiapprox} shows that the combination of (\ref{riereas2}) and  (\ref{riereas3}) gives an increasingly accurate result as $w$ increases beyond 0.8. Knowledge of only the first term in (\ref{riereas2}) is sufficient to indicate that the prefactor contribution increasingly dominates the zeta-function contribution as $w$ tends up towards unity. Hence, it is evident that $\log \xi(1/(1-w))$ is finite
on the real axis of $w$ for every value below unity, so that the radius of convergence of its power series about $w=1/2$ cannot be other than one half. This forms a simple argument for the location of all zeros of $\zeta (s-1/2)$ being on $\sigma=1$.
 
\section{Expansions in numerical form}
\subsection{Power series in the variable $s$}
In the following, we will exhibit both analytic expressions and numerical forms of power series for various functions. The former will be limited by the rapid increase of complexity of coefficient expressions with order. Firstly, for $\xi (s+1/2)$,
\begin{equation}
\xi(s+1/2)=C_0+C_1 s+C_2 s^2+C_3 s^3+C_4 s^4+\ldots,
\label{bpust9}
\end{equation}
where
\begin{equation}
C_0=-\frac{\zeta \left(\frac{1}{2}\right) \Gamma \left(\frac{1}{4}\right)}{8 \sqrt[4]{\pi
   }},
 \label{bpust10}
\end{equation}
and
\begin{equation}
C_1=0=\frac{\Gamma \left(\frac{1}{4}\right) \left(-2 \zeta
   '\left(\frac{1}{2}\right)+\zeta \left(\frac{1}{2}\right) \log (\pi )-\zeta
   \left(\frac{1}{2}\right) \psi ^{(0)}\left(\frac{1}{4}\right)\right)}{16 \sqrt[4]{\pi
   }} .
\label{bpust11}
\end{equation}
From 
(\ref{bpust11}) we deduce
\begin{equation}
\zeta'\left(\frac{1}{2}\right)=\frac{1}{2} \zeta \left(\frac{1}{2}\right)\left[ \log (\pi)-\psi^{(0)}\left(\frac{1}{4}\right) \right].
\label{bpust12}
\end{equation}
In this way, all odd order derivatives of $\zeta (s)$ at $s=1/2$ may be eliminated from power series. This gives:
\begin{equation}
C_2=\frac{\Gamma \left(\frac{1}{4}\right) \left(\zeta \left(\frac{1}{2}\right)
   \left(32+\log ^2(\pi )+\psi ^{(0)}\left(\frac{1}{4}\right)^2-\psi
   ^{(1)}\left(\frac{1}{4}\right)-2 \log (\pi ) \psi
   ^{(0)}\left(\frac{1}{4}\right)\right)-4 \zeta ''\left(\frac{1}{2}\right)\right)}{64
   \sqrt[4]{\pi }} .
   \label{bpust13}
   \end{equation}
  Numerically, the even order  coefficients up to order 10 are:\\
   (0, 0.497120778188314109912774), (2, 
  0.0114859721575727187676249), \\ (4, 
  0.000123452018070318006890346),
   (6, 
  8.32355481385527072004759  $10^{-7}$), \\ (8, 
  3.99222655134413717472527 $10^{-9}$), 
  (10, 1.46160257601109608624121 $10^{-11}$).
  
Next, for $\xi (s-1/2)$,
\begin{equation}
\xi(s-1/2)=D_0+D_1 s+D_2 s^2+D_3 s^3+D_4 s^4+\ldots ,
\label{bpust9m}
\end{equation}
where
\begin{equation}
D_0=\frac{3}{8} \pi^{1/4} \Gamma( -1/4) \zeta(-1/2), 
\label{bpust10m}
\end{equation}
and
\begin{equation}
D_1=\frac{1}{16}  \pi^{1/4} \Gamma(-1/4)  [-16\zeta(-1/2) - 
   3  \log(\pi) \zeta(-1/2)+ 
   3 \psi_0 (-1/4) \zeta(-1/2) + 
   6 \zeta'(-1/2)].
\label{bpust11m}
\end{equation}
The expressions for $D_2$, $D_3$ and $D_4$ are too complicated to warrant reproduction.
Numerically, the coefficients up to order 10 are:\\
0.508731038726323958025671, -0.0234707786048020825988372, \\
0.0122392820411106099993383, -0.000510680509960582081197381, \\
0.000136219896777660495235434, -5.22141487535619756401486 $10^{-6}$, \\
9.47246998269225412163851 $10^{-7}$, -3.37259460993816995307624 $10^{-8}$, \\
4.67141900041698784170513 $10^{-9}$, -1.55773064531108593374648 $10^{-10}$, \\
1.75421192426233609057205 $10^{-11}$.

Combining (\ref{bpust9}) and  (\ref{bpust9m}),
\begin{equation}
\xi_\pm (s)=E_0^\pm+E_1^\pm s+E_2^\pm s^2+E_3^\pm s^3+E_4^\pm s^4+\ldots.
\label{combpust1}
\end{equation}
For the function $\xi_+(s)$, the numerical values of the first ten coefficients $E_n^+$  are:\\
0.502925908457319033969223, -0.0117353893024010412994186, \\
0.0118626270993416643834816, -0.000255340254980291040598691, \\
0.000129835957423989251062890, -2.61070743767809878200743 $10^{-6}$, \\
8.89801239827376242084305 $10^{-7}$, -1.68629730496908497653812 $10^{-8}$, \\
4.33182277588056250821520 $10^{-9}$, -7.78865322655542966873238 $10^-{11}$, \\
1.60790725013671608840663 $10^{-11}$.\\
For the function $\xi_-(s)$, the first ten coefficients $E_n^-$  are:\\
-0.005805130269004924056449, 0.0117353893024010412994186, \\
-0.0003766549417689456158567, 0.000255340254980291040598691, \\
-6.383939353671244172544 $10^{-6}$, 2.61070743767809878200743 $10^{-6}$, \\
-5.7445758441849170079546 $10^{-8}$, 1.68629730496908497653812 $10^{-8}$, \\
-3.3959622453642533348993 $10^{-10}$, 7.78865322655542966873238 $10^{-11}$, \\
-1.4630467412562000216542 $10^{-12}$ .
\subsection{Numerical results related to those of Keiper}
We now consider relationship between power series  expansions and sums of inverse powers of zeros, building on work of Keiper \cite{keiper} and Lehmer \cite{lehmer}. We start with the expansion of $\xi(s)$ about $s=0$:
\begin{equation}
\xi(s)=F_0+F_1 s+F_2 s^2+F_3 s^3+F_4 s^4+\ldots ,
\label{bpust9k}
\end{equation}
where
\begin{equation}
F_0=\frac{1}{2}, ~F_1=\frac{1}{4} (-2 - \gamma + \log(4 \pi)), 
\label{bpust10k}
\end{equation}
and
\begin{equation}
F_2=\frac{1}{32} (8 \gamma - 6 \gamma^2 + \pi^2 + 
   2 \log(4 \pi) (-4 - 2 \gamma + \log (4 \pi )) - 
   16 \gamma_1) ,
\label{bpust11k}
\end{equation}
where $\gamma_1\approx -0.0728158$ is the Stieltjes gamma constant of order unity. The numerical values of the first eleven coefficients $F_n$ are:\\
0.500000000000000000000000, -0.0115478544830605169071551, \\
0.0116719322671130915674412, -0.000248991924961474336175586, \\
0.000126590865158263502528061, -2.52512739610958708479262 $10^{-6}$, \\
8.60493520930767788890077 $10^{-7}$, -1.61892073094053848017403 $10^{-8}$, \\
4.15798412501386081535438 $10^{-9}$, -7.42620960745947002261292 $10^{-11}$, \\
1.53278011638166567551397 $10^{-11}$.

We obtain for $\sigma_1$ to $\sigma_{10}$:\\
0.0230957089661210338143102, -0.0461543172958046027571080, \\
-0.00011115823145210592276267, 0.00007362722126168951832677, \\
7.15093355762607735801 $10^{-7}$, -2.81436416938766261607 $10^{-7}$, \\
-4.5741911497047721112 $10^{-9}$, 1.26886811095076071901 $10^{-9}$, \\
2.8274371550558870893$10^{-11}$, -5.997714847151874595 $10^{-12}$.

Keiper gives a table to $\sigma$ values accurate to 40 decimal places. Those given here differ by at most one in the last figure.
 \subsection{Numerical Series in $w$}  
 
 The numerical form of the Taylor series of  $\xi (1/(1-w)+1/2$) to order 12 is:\\ 
 \begin{eqnarray}
 \xi \left( \frac{1}{1-w}+\frac{1}{2}\right)&=& 0.508731038726323958025671+ 0.0234707786048020825988372 w  \nonumber\\ &&+0.0357100606459126925981755 
    w^2+0.0484600231969838846787112 w^3 \nonumber\\
&&    +0.0618568861547933193356797
   w^4+0.0760420908309940132618804 w^5  \nonumber \\
   &&  +0.0911632471991126085730891
   w^6+0.107375114867493741415169 w^7 \nonumber \\
   && +0.124840622449609510369871
   w^8
   +0.143731930158926109607415 w^9 \nonumber \\
   && +0.164231540628834875507686
   w^{10}+0.186533463149563699312568 w^{11}\nonumber \\
 &&  +0.210844436724090899321712
   w^{12}+O\left(w^{13}\right).
    \label{bpust1wn}
 \end{eqnarray}
 
  This corresponding  numerical form for $\xi (1/(1-w)-1/2)$ to order 12 is:\\ 
 \begin{eqnarray}
 \xi \left( \frac{1}{1-w}-\frac{1}{2}\right)&=& 0.497120778188314109912774+0.011485972157572718767622 w^2 \nonumber\\
& &+0.022971944315145437535244  w^3+0.034581368490788474309757 w^4 \nonumber\\
& & +0.046437696702572147098050  w^5+0.058665213324048159434086 w^6 \nonumber\\
& & +0.071389867439730985905972  w^7+0.084740109192805809027167 w^8  \nonumber\\
& & +0.098847734117289558795967  w^9+0.113848739461487632047511 w^{10}\nonumber\\
& & +0.12988419653882092022278  w^{11}+0.14710114318598859685984 w^{12}\nonumber \\
&& +O\left(w^{13}\right).
  \label{bpust1wan}
 \end{eqnarray}

The series for the logarithms of (\ref{bpust1wn}) and (\ref{bpust1wan}) are:
\begin{eqnarray}
\log \left[ \xi \left( \frac{1}{1-w}+\frac{1}{2}\right) \right]&=&
-0.675835813236695767842275+0.0461359280604625753594660 w \nonumber \\
& & +0.0691301196352103328072490
   w^2+0.0920509179650365681271473 w^3 \nonumber \\
& &   +0.114880446150576506783275
   w^4+0.137601106517779960377133 w^5 \nonumber \\
& &   +0.160195624167229059202861
   w^6+0.182647089272462999988600 w^7 \nonumber \\
   & & +0.204938997989198298868628
   w^8+0.227055291843200263329215 w^9 \nonumber \\
& &   +0.248980395470907623248291
   w^{10}+0.270699252593709637403009 w^{11} \nonumber \\
& &   +0.292197360113993960590683
   w^{12}+O\left(w^{13}\right),
   \label{lxihser}
   \end{eqnarray}
   and
   \begin{eqnarray}
\log \left[ \xi \left( \frac{1}{1-w}-\frac{1}{2}\right) \right]&=&-0.698922267945331415298362+0.0231049931154189707889338 w^2 \nonumber \\
&& +0.0462099862308379415778676
   w^3+0.0692963930466142775237190 w^4 \nonumber \\
&&   +0.0923456272631053437834055
   w^5+0.115339150638645639171605 w^6 \nonumber \\
&&   +0.138258521047523929818514
   w^7+0.161085440372202462764362 w^8 \nonumber \\
&&   +0.183801802064020339427427
   w^9+0.206389738207265861712433 w^{10} \nonumber \\
&&   +0.228831665922788129186448
   w^{11}+0.251110332949243718084027 w^{12} \nonumber \\
&&   +O\left(w^{13}\right).
 \label{lximser}
   \end{eqnarray}
   
   The series for the logarithms of (\ref{bpust2w}) and (\ref{bpust2wa}) are
    \begin{eqnarray}
    \log \left[ \xi_+\left( \frac{1}{1-w} \right) \right] &=&
   -0.687312419021627833700725+0.023334230957395524232106 w \nonumber \\
   && +0.046649213860405348103991
   w^2+0.069925751526272229957378 w^3  \nonumber \\
   &&+0.093144748336397953034365
   w^4+0.116287260571386560139602 w^5 \nonumber \\
   && +0.139334546211624749855289
   w^6+0.162268114028537401476190 w^7 \nonumber \\
 &&  +0.18506977179446033912675
   w^8+0.20772167344255548785287 w^9 \nonumber \\
 &&  +0.23020636501234089942540
   w^{10}+0.25250682922120233406310 w^{11} \nonumber \\
   &&+0.27460652850767407018678
   w^{12}+O\left(w^{13}\right) ,
   \label{logxipsern}
   \end{eqnarray}
   and 
     \begin{eqnarray}
    \log \left[ \xi_-\left( \frac{1}{1-w} \right) \right] &=&
    -5.1490132232563522103123+2.021554858994170669692 w\nonumber \\
    & & +0.04309595122735336775
   w^2+0.73127620503659101949 w^3 \nonumber \\
  & &  +0.0860819388755765953 w^4+0.5074995561747003636
   w^5 \nonumber \\
  & &  +0.1288489066305341210 w^6+0.435830932894465061 w^7 \nonumber \\
  & &  +0.171289602583570808
   w^8+0.414577014426957098 w^9 \nonumber \\
   & & +0.21329945625114847 w^{10}+0.41592926321690695
   w^{11} \nonumber \\
   & & +0.25477742573189697 w^{12}+O\left(w^{13}\right)  .
   \label{logximsern}
   \end{eqnarray}

\end{document}